\documentclass[a4paper]{article}
\usepackage{amsthm,amsfonts,amsmath,amssymb}
\usepackage[cp1251]{inputenc}
\usepackage[english]{babel}
\usepackage[final]{graphicx}
\usepackage{setspace}
\usepackage[12pt]{extsizes}
\begin{document}
\newcommand{\per}{{\rm per}}
\newtheorem{teorema}{Theorem}
\newtheorem{lemma}{Lemma}
\newtheorem{utv}{Proposition}
\newtheorem{svoistvo}{Property}
\newtheorem{sled}{Corollary}
\newtheorem{con}{Conjecture}

\author{A. A. Taranenko}
\title{Upper bounds on the permanent of multidimensional (0,1)-matrices \thanks{The work is supported
by the Russian Science Foundation (grant 14--11--00555)}}
\date{}

\maketitle

\begin{abstract}
The permanent of a multidimensional matrix is the sum of products of entries over all diagonals.

By Minc's conjecture,  there exists a reachable upper bound on the permanent of 2-dimensional (0,1)-matrices. In this paper we obtain some generalizations of  Minc's conjecture to the multidimensional case. For this purpose we prove and compare several bounds on the permanent of multidimensional (0,1)-matrices. 

Most estimates can be used for matrices with nonnegative bounded entries.
\end{abstract}

\section{Definitions and upper bounds on the permanent of 2-dimensional matrices}

Let $n,d \in \mathbb N$, and let $I_n^d= \left\{ (\alpha_1, \ldots , \alpha_d):\alpha_i \in \left\{1,\ldots,n \right\}\right\}$.
A \textit{$d$-dimensional matrix $A$ of order $n$} is an array $(a_\alpha)_{\alpha \in I^d_n}$, $a_\alpha \in\mathbb R$. A matrix $A$ is called \textit{nonnegative} if $a_\alpha \geq 0$ for all $\alpha \in I^d_n$.

Let $k\in \left\{0,\ldots,d\right\}$. A \textit{$k$-dimensional plane} in $A$ is the set of entries obtained by fixing $d-k$ indices and letting the other $k$ indices vary from 1 to $n$.  A $(d-1)$-dimensional plane is said to be a \textit{hyperplane}.  The \textit{direction} of a plane is the vector describing which indices are fixed in the plane.

Let $\alpha$ belong to $I$, and let $(A|\alpha)$ denote the $d$-dimensional matrix of order $n-1$ obtained from the matrix $A$ by deleting the entries $a_\beta$ such that $\alpha_i=\beta_i$ for some $i \in \left\{1,\ldots,d \right\}$.

For a $d$-dimensional matrix $A$ of order $n$, denote by $D(A)$ the set of its diagonals
$$D(A)=\left\{ (\alpha^1,\ldots,\alpha^n) | \alpha^i \in I_n^d, \forall i ~\forall j \neq i ~ \rho (\alpha^i,\alpha^j)=d\right\},$$
where $\rho$ is the Hamming distance (the number of positions at which the corresponding entries are different).
Then the \textit{permanent} of a matrix $A$ is 
$$\per A = \sum\limits_{p\in D} \prod\limits_{\alpha \in p} a_\alpha.$$

In this paper we mostly consider (0,1)-matrices, that is, matrices  all of whose entries are equal to 0 or 1. But sometimes we concern matrices with nonnegative entries, which are not greater than 1.

First we show the trivial upper bound on the permanent of nonnegative 2-dimensional matrices.

\begin{utv} \label{trivoc}
Let $A$ be a nonnegative 2-dimensional matrix of order $n$. Suppose that the sum of entries in the $i$th row of the matrix $A$ is not grater than $r_i$. Then
$$\per A \leq \prod \limits_{i=1}^n r_i.$$
\end{utv} 

\begin{proof}
The proof is by induction on the order of matrices. Using the definition of the permanent, we have
$$ \per A = \sum \limits_{j=1}^n a_{n,j} \per (A|(n,j)) \leq r_n \max_{j= 1, \ldots, n} \per (A|(n,j)).$$

Note that $(A|(n,j))$ are the matrices of order $n-1$  such that  the sum of entries in their $i$th row is not greater than $r_i$. By the inductive assumption, $\per (A|(n,j)) \leq \prod \limits_{i=1}^{n-1} r_i$ for all $j=1 \ldots n$. Therefore,
$$\per A \leq \prod \limits_{i=1}^n r_i.$$
\end{proof}
 
The following inequality, proved by Bregman~\cite{bregman}, Schrijver~\cite{shriver}, and
Radhakrishnan~\cite{krishna}, is known as Minc's conjecture~\cite{gypminc}.

\begin{teorema}[\cite{bregman,krishna,shriver}] \label{gminc}
Let $A$ be a 2-dimensional matrix of order $n$, and let $r_i$ be the number of 1's in the $i$th row of the matrix $A$. Then
$$\per A \leq \prod \limits_{i=1}^n r_i!^{1/r_i}.$$ 
\end{teorema}
Using the theorem, we can extend the inequality on the permanent of nonnegative matrices with bounded entries.

\begin{sled} \label{mincsled}
Let $A$ be a nonnegative 2-dimensional matrix of order $n$ whose entries are not greater than 1. Suppose that the sum of entries in the $i$th row of $A$ is not greater than $r_i$. Then
$$\per A \leq \prod \limits_{i=1}^n \left\lceil r_i\right\rceil!^{\frac{1}{\left\lceil r_i\right\rceil}}.$$
\end{sled}

\begin{proof}
Construct recursively nonnegative $2$-dimensional matrices  $A=A^0, A^1, \ldots, A^n$ such that their entries are not greater than 1 and $\per A^i \leq \per A^{i+1}$ for all $i \in \left\{0, \ldots, n-1\right\}$. 

Assume that the matrix $A^i$ is constructed. Let us construct $A^{i+1}.$ Rearrange the columns of the matrix $A^i$ so that $\per (A^i|(i+1,k)) \geq \per(A^i|(i+1,k+1))$ for all $k$. Call the resulting matrix $B^i$. Let $A^{i+1}= (a^{i+1}_{j,k})_{j,k=1}^n$ and $B^{i}= (b^{i}_{j,k})_{j,k=1}^n.$
Put 

$a^{i+1}_{j,k} = b^i_{j,k}$ for $j \neq i+1$,

$a^{i+1}_{i+1,k} = 1$ for $k \leq \left\lceil \gamma_{i+1} \right\rceil$,

and $a^{i+1}_{i+1,k} = 0$ for $k > \left\lceil \gamma_{i+1} \right\rceil$.

Then $A^n$ is a (0,1)-matrix with $\left\lceil r_i\right\rceil$ ones in the $i$th row.  By construction, $\per A^i = \per B^i \leq \per A^{i+1}.$  By Theorem~\ref{gminc}, we have 
$$\per A \leq \per A^n \leq \prod\limits_{i=1}^n \left\lceil r_i\right\rceil ! ^{\frac{1}{\left\lceil r_i\right\rceil}}.$$
\end{proof}

If we know only the sum of all entries of a 2-dimensional nonnegative matrix, then we can estimate its permanent by the following inequality.

\begin{sled}
Let $A$ be nonnegative 2-dimensional matrix of order $n$ whose entries are not greater than 1. Suppose that $\sum\limits_{i,j=1}^n a_{i,j}= \gamma n$. Then
$$\per A \leq (\gamma+1)^n e^{-n}  (e\sqrt{ \gamma+1})^{\frac{n}{ \gamma+1}}.$$
\end{sled}

\begin{proof}
Suppose that the sum of entries in the $i$th row of $A$ is equal to $r_i$. By Corollary~\ref{mincsled}, 
$$\per A \leq \prod \limits_{i=1}^n \left\lceil r_i\right\rceil!^{\frac{1}{\left\lceil r_i\right\rceil}}.$$

Note that $\sum\limits_{i=1}^n \left\lceil r_i\right\rceil \leq \sum\limits_{i=1}^n ( r_i +1 ) = ( \gamma +1) n.$ Using an approximation of the factorial
$$x! \leq e x^{x+1/2} e^{-x},$$
we obtain
$$\per A \leq \prod\limits_{i=1}^n e^{-1 + 1/\left\lceil r_i\right\rceil} \left\lceil r_i\right\rceil^{1 + \frac{1}{2\left\lceil r_i\right\rceil}}.$$

It can be proved that $e^{1/x} x^{1+1/2x}$ is a concave function for $x>1$. Therefore,
$$\per A \leq \prod\limits_{i=1}^n e^{-1 + \frac{1}{\gamma+1}} ( \gamma +1)^{1 + \frac{1}{2( \gamma+1)}} =
( \gamma+1)^n e^{-n}  (e\sqrt{ \gamma+1})^{\frac{n}{ \gamma+1}}. $$

\end{proof}

In the following section we prove upper bounds on the permanent of multidimensional (0,1)-matrices through the number of planes covering all ones of the matrix. Also, we propose an upper bound by means of sums of entries in hyperplanes and prove that it holds asymptotically.  In addition, we estimate the permanent of a 3-dimensional (0,1)-matrix through the permanent of some 2-dimensional matrix.

\section{Upper bounds on the permanent of multidimensional matrices}

There is a trivial upper bound on the permanent of nonnegative multidimensional matrices, which is similar to Proposition~\ref{trivoc}.

\begin{utv} \label{mtriv}
Let $A$ be a $d$-dimensional matrix of order $n$.  Suppose that the sum of entries in the $i$th hyperplane of the matrix $A$ is not grater than $r_i$. Then
$$\per A \leq \prod \limits_{i=1} ^n r_i.$$
\end{utv}

Consider multidimensional (0,1)-matrices now. In~\cite{my}, the author proved an asymptotic upper bound on the permanent of matrices such that each  1-dimensional plane of the matrix contains exactly one 1.

\begin{teorema}
Let $d \geq 3$, and let $\Omega^d(n)$ be the set of $d$-dimensional (0,1)-matrices of order $n$ such that each  1-dimensional plane contains exactly one 1. Then
$$\max \limits_{A \in \Omega^d(n)} \per A \leq n!^{d-2} e^{-n + o(n)} \mbox{ as } n \rightarrow \infty.$$

\end{teorema}

It would be great to generalize the bound from Theorem~\ref{gminc} to the multidimensional case and to get a bound on the permanent in terms of sums in hyperplanes. But at the moment we can estimate the permanent in terms of the number of planes covering all ones in a matrix. To make the further reasoning clearer, we prove a simple lemma.

\begin{lemma} \label{vsp}
Let $A$ be a $d$-dimensional matrix of order $n$. Let us fix some direction of $k$-dimensional planes, $1 \leq k \leq d-2,$ and enumerate them by $(d-k)$-dimensional indices. Put
$$T = \left\{ \tau: I^1_n \rightarrow I^{d-k - 1}_n | ((1, \tau(1)), \ldots, (n, \tau(n)))  \mbox{ is a diagonal in a } (d-k)\mbox{-dimensional matrix}\right\}.$$

 Denote by  $A_\tau$  the $(k+1)$-dimensional matrix of order $n$ such that the $i$th hyperplane of $A_\tau$ is the $(i,\tau(i))$-th $k$-dimensional plane of the matrix $A$. Then
$$\per A = \sum\limits_{\tau \in T} \per A_\tau.$$
\end{lemma}

\begin{proof}

Without loss of generality we suppose that in the $k$-dimensional planes the first $d-k$ indices are fixed and the last $k$ indices vary. By the definition, 
$$\per A = \sum\limits_{p\in D} \prod\limits_{\alpha \in p} a_\alpha.$$

Divide the set $D$ of diagonals onto the parts
$$D_\tau = \left\{p \in D| p=((1, \tau(1), *, \ldots,*), \ldots, (n, \tau(n),*,\ldots,*))\right\},$$
where $*$ means an arbitrary symbol. Rearrange the summands in the definition of the permanent
$$\per A = \sum\limits_{\tau \in T}\sum\limits_{p\in D_\tau} \prod\limits_{\alpha \in p} a_\alpha.$$
 
Since $\sum\limits_{p\in D_\tau} \prod\limits_{\alpha \in p} a_\alpha$ is the permanent of the matrix $A_\tau$, the proof is over.

\end{proof}
Let $A$ be a $d$-dimensional matrix of order $n$.  Denote by $L ^ k_i (A)$ the set of $k$-dimensional planes in $A$ such that their last $k$ indices vary, the first $d-k$ indices are fixed, and the very first index equals $i$.
Let us prove an upper bound on the permanent of multidimensional matrices.

\begin{teorema} \label{bad}
Let $A$ be a $d$-dimensional (0,1)-matrix of order $n$. Suppose that all ones in the $i$th hyperplane $L ^ {d-1}_i (A)$  can be covered by  $s_{i,d-1}$  planes from  $L ^ {d-2}_i (A)$, $\ldots$, all ones in a plane from $L ^ {k}_i (A)$ can be covered by $s_{i,k}$  planes from  $L ^ {k-1}_i (A)$, $\ldots$, each 1-dimensional plane from  $L ^ 1_i (A)$ contains   $s_{i,1}$ ones at most. Then
$$\per A \leq \prod \limits_{k=1}^{d-1} \prod \limits_{i=1}^n  s_{i,k}!^{1/s_{i,k}}.$$
\end{teorema}

\begin{proof}
The proof is by induction on the dimension of matrices. We consider only the first step of induction, that is the step from the 2-dimensional case to the 3-dimensional case.

Let $A$ be a 3-dimensional (0,1)-matrix of order $n$, and let $l_{i,j}$ be 1-dimensional planes in the $i$th hyperplane of the matrix $A$. Suppose that there are $m_i$ 1-dimensional planes $l_{i,j}$ containing ones. Also assume that each plane $l_{i,j}$  contains at most $s_i$ ones.

Put $S= \left\{ \sigma \in S_n |  \mbox{plane } l_{i,\sigma (i)} \mbox{ contains ones for all }  i\right\},$ where $S_n$ is the symmetric group on $\left\{1, \ldots, n \right\}$. 
By Lemma~\ref{vsp},
$$\per A = \sum \limits_{\sigma \in S} \per A_{\sigma},$$
where  $A_{\sigma}$  is the 2-dimensional (0,1)-matrix such that its $i$th row is the plane  $l_{i, \sigma (i)}.$

Note that for all $\sigma \in S$, there are  $s_i$  ones in the $i$th row of the matrix $A_\sigma$. By Theorem~\ref{gminc}, we have
$$\per A_\sigma \leq \prod \limits_{i=1}^n s_i!^{1/s_i}$$
for all $\sigma \in S.$ Consequently,
$$\per A \leq |S| \prod \limits_{i=1}^n s_i!^{1/s_i}.$$

Estimate the cardinality of the set $S$ now. For this purpose, consider the (0,1)-matrix $B$ such that  $b_{i,j} = 1$ if and only if the plane $ l_{i,j}$ contains ones. Notice that $|S| = \per B$  and that the $i$th row of $B$ contains $m_i$ ones. Using Theorem~\ref{gminc}, we obtain
$$\per B \leq \prod \limits_{i=1}^n m_i!^{1/m_i}.$$

Therefore,
$$\per A \leq \prod \limits_{i=1}^n m_i!^{1/m_i} s_i!^{1/s_i}.$$
\end{proof}

The equality holds, for example, if the matrix $A$ is a block diagonal matrix. The bound will be rough in many cases, because it depends on the arrangement of ones in the matrix. Unfortunately, we have no success in generalization of the known proofs of Theorem~\ref{gminc} to the multidimensional case, and we don't have a good estimate on the permanent of matrices through the sums in hyperplanes. But we propose the following conjecture, which was tested on a number of matrices of small order and dimension.

\begin{con}\label{voc}
Let $A$ be a $d$-dimensional (0,1)-matrix of order $n$. Suppose that there are $r_i$ ones in the $i$th hyperplane of $A$. Then
$$\per A \leq n!^{d-2} \prod \limits_{i=1}^n \left\lceil \frac{r_i}{n^{d-2}}\right\rceil ! ^{ \frac{1}{\left\lceil r_i /n^{d-2}\right\rceil}}.$$
\end{con}

This equality holds if all 2-dimensional planes of some direction in the matrix $A$ are a 2-dimensional matrix such that the equality from Theorem~\ref{gminc} holds on it.

If Conjecture~\ref{voc} is true, we can estimate the permanent of matrices by means of the sums in planes of an arbitrary dimension. 

Actually, let $A$ be a $d$-dimensional (0,1)-matrix of order $n$, and let $l_\beta$ be $k$-dimensional planes of some direction in the matrix $A$. Suppose that the sum of entries in the plane $l_\beta$ is equal to $r_\beta$. By Lemma~\ref{vsp},
$$\per A = \sum \limits_{\tau \in T} \per A_\tau,$$
where 
$$T = \left\{ \tau: I^1_n \rightarrow I^{d-k - 1}_n | ((1, \tau(1)), \ldots, (n, \tau(n)))  \mbox{ is a diagonal in a } (d-k)\mbox{-dimensional matrix}\right\}.$$
Note that $A_\tau$ is the $(k+1)$-dimensional matrix such that $l_{i,\tau (i)}$ are its hyperplanes. The sum of entries in the hyperplanes of the matrix  $A_\tau$ is equal to $r_{i,\tau (i)}.$  If the conjecture is true, then $\per A_\tau \leq n!^{k-1} \prod \limits_{i=1}^n h_{i,\tau (i)}!^{1/h_{i,\tau (i)}},$ where $h_{i,\tau (i)} = \left\lceil \frac{r_{i,\tau (i)}}{n^{k-1}}\right\rceil.$ Therefore,
$$\per A \leq n!^{k-1} \sum \limits_{\tau \in T} \prod \limits_{i=1}^n h_{i,\tau (i)}!^{1/h_{i,\tau (i)}}.$$
Let $B$ be a $d$-dimensional matrix of order $n$ such that $b_\beta = h_{\beta}!^{1/h_{\beta}}$, $\beta \in I_n^{d-k}.$  Then
$$\per A \leq n!^{k-1} \per B.$$

If  $l_\beta$  are 1-dimensional planes, then we can  estimate the permanent of $A_\tau$ with the help of  Theorem~\ref{gminc}. Therefore the following proposition holds.
\begin{utv} \label{groc}
Let $A$ be a $d$-dimensional (0,1)-matrix of order $n$, and let $l_\beta$ be 1-dimensional planes of some direction in the matrix $A$. Suppose that the sum of entries in the plane $l_\beta$ is equal to $r_\beta$. Consider the $(d-1)$-dimensional matrix $B$ of order $n$ such that $b_\beta = r_{\beta}!^{1/r_{\beta}}.$ Then
$$\per A \leq  \per B.$$
\end{utv}

The following example illustrates that this bound is weaker than Conjecture~\ref{voc}. Consider the  3-dimensional (0,1)-matrix of order 3:
$$
A = \left( \begin{array} {cccc}
1 & 1 & 1 & 1 \\ 1 & 1 & 0 & 0 \\ 1 & 0 & 0 & 0 \\ 1 & 0 & 0 & 0
\end{array}
\times
\begin{array} {cccc}
0 & 0 & 1 & 1 \\ 1 & 1 & 1 & 1 \\ 0 & 0 & 1 & 0 \\ 0 & 0 & 1 & 0
\end{array}
\times
\begin{array} {cccc}
0 & 1 & 0 & 0 \\ 0 & 1 & 0 & 0 \\ 1 & 1 & 1 & 1 \\ 1 & 1 & 0 & 0
\end{array}
\times
\begin{array} {cccc}
0 & 0 & 0 & 1 \\ 0 & 0 & 0 & 1 \\ 0 & 0 & 1 & 1 \\ 1 & 1 & 1 & 1
\end{array}
\right).
$$
Its permanent equals 74, the sum of entries in each hyperplane equals 8. Consider the matrix $B$ such that $b_{i,j}= s_{i,j}!^{1/s_{i,j}}$, where $s_{i,j}$ is the sum of entries from  $(i,j)$-th row of the matrix $A$. It can be checked that  $\per B \approx 104,23.$ But Conjecture~\ref{voc} claims that $\per A \leq 4! * 2^{4/2} = 96.$ On the other hand, Theorem~\ref{bad} gives that $\per A \leq 4!^2 = 576.$

However we can prove that Conjecture~\ref{voc} holds asymptotically for matrices such that the number of ones in its hyperplanes is sufficiently large.

\begin{teorema} \label{asym}
Assume that for some integer $d \geq 2$ and for all integer $n$ there are $n$ integer numbers  $r_1(n), \ldots ,r_n(n)$ such that $\min\limits_{i=1 \ldots n} r_i(n)/ n^{d-2} \rightarrow \infty$ as $n \rightarrow \infty.$ Denote by $\Lambda^d(n,r)$  the set of $d$-dimensional (0,1)-matrices of order $n$ such that the number of ones in their hyperplanes is not greater than $r_i(n).$ Put $F(x) = \left\lceil x\right\rceil!^{1/\left\lceil x\right\rceil}.$ Then
$$ \max\limits_{A \in \Lambda^d(n,r)} \per A \leq n!^{d-2} e^{o(n)} \prod \limits_{i=1}^n F\left(\frac{r_i(n)}{n^{d-2}}\right) \mbox{ as } n \rightarrow \infty.$$
\end{teorema}

\begin{proof}
The proof is by induction on the dimension of matrices. The basis of induction is the case $d=2$, that holds by Thereom~\ref{gminc}. 

Assume that for dimension $d-1$ the theorem holds, and prove the theorem for dimension $d$. Let $A$ be an arbitrary matrix from $\Lambda^d(n,r).$ Then the number of ones in its $i$th hyperplane is not greater than $r_i(n).$ 

Divide these hyperplanes into 1-dimensional planes $l_\beta$ of some direction.
Assume that there are  $s_\beta$ ones in the plane $l_\beta$. Consider the $(d-1)$-dimensional matrix $B$ such that $b_\beta= s_\beta!^{1/s_\beta}$.
By Proposition~\ref{groc}, $\per A \leq \per B.$ Put $f(x)= x!^{1/x}.$  Using an approximation of the factorial, we estimate $f(x):$ 
$$(2\pi x)^{1/2x} x e^{-1} \leq f(x) \leq e^{1/x} x^{1+1/2x} e^{-1}.$$
Denote by  $g(x)$  the right-hand side of the inequality. 
It can be checked that $g(x)$ is a concave function. Since  $\sum \limits_{\beta = (i,*,\ldots,*)} s_\beta = r_i(n)$, it follows that 
$$\sum \limits_{\beta = (i,*,\ldots,*)} g(s_\beta) \leq n^{d-2} g(r_i(n)/n^{d-2}),$$
where $*$ means an arbitrary symbol.
Estimate the sum of entries in the $i$th hyperplane of $B$:
$$\sum \limits_{\beta = (i,*,\ldots,*)} f(s_\beta) \leq n^{d-2}g(r_i(n)/n^{d-2}) = r_i(n) e^{-1+n^{d-2}/r_i(n)} \left(\frac{r_i(n)}{n^{d-2}}\right)^{n^{d-2}/2r_i(n)}.$$

Note that the entries of the matrix $B$ are not greater than $f(n).$ As in the proof of Corollary~\ref{mincsled}, we rearrange the sum of entries in the  hyperplanes of $B$ and obtain the $(d-1)$-dimensional matrix $C$ of order $n$ such that the entries of $C$ equal 0 or $f(n)$ and $\per B \leq \per C.$ Note that there are at most  $\frac{n^{d-2}g(r_i(n)/n^{d-2})}{f(n)}$ nonzero entries in the $i$th hyperplane of the matrix $C$.

Using the inequality for $f(n)$, we obtain an upper bound on the number of nonzero entries in the $i$th hyperplane of $C$
$$\frac{n^{d-2}g(r_i(n)/n^{d-2})}{f(n)} \leq \frac{e^{n^{d-2}/r_i(n)} }{(2\pi n)^{1/2n} } \left( \frac{r_i(n)}{n^{d-2}}\right)^{n^{d-2}/2r_i(n)} \frac{r_i(n)}{n}.$$
Denote by $s_i(n)$  the right-hand side of the inequality.

Let $\Lambda^{d-1}(n,s)$  be the set of $(d-1)$-dimensional (0,1)-matrices of order $n$ such that the number of ones in their $i$th hyperplane is not greater than $s_i(n).$ If we divide all entries of the matrix $C$ by $f(n)$, we obtain some matrix from $\Lambda^{d-1}(n,s)$.

Since $\per B \leq \per C$ and $\per A \leq \per B,$ it follows that $\per A \leq \per C.$   Because $A$ is an arbitrary matrix from $\Lambda^{d-1}(n,s)$, we have
$$\max\limits_{A \in \Lambda^d(n,r)} \per A \leq f(n)^n \max\limits_{C \in \Lambda^{d-1}(n,s)} \per C.$$

Under the hypothesis of the theorem, we get that $s_i(n) = (r_i(n)/n)^{1+o(1)}$ and  $\min\limits_{i=1 \ldots n} s_i(n)/ n^{d-3} \rightarrow \infty$ as $n\rightarrow\infty$. Therefore $\Lambda^{d-1}(n,s)$ satisfies the conditions of the theorem. By the inductive assumption, we finally obtain

$$\max\limits_{A \in \Lambda^d(n,r)} \per A  \leq  f(n)^n n!^{d-3} e^{o(n)}  \prod \limits_{i=1}^n F\left(\left(\frac{r_i(n)}{n^{d-2}}\right)^{1+ o(1)}\right) = n!^{d-2} e^{o(n)} \prod \limits_{i=1}^n F\left(\frac{r_i(n)}{n^{d-2}}\right) \mbox{ as } n \rightarrow \infty .$$

\end{proof}

By the same argument as in Corollary~\ref{mincsled}, we can obtain the similar results for nonnegative matrices with bounded entries.

\end{document}